\documentclass{birkjour}
\usepackage{amsmath,amsfonts,amssymb}
\usepackage{amsthm}
\usepackage{mathtools}
\usepackage{graphics,graphicx,epsf}
\usepackage{hyperref}
\usepackage{float}
\usepackage{color}

\newtheorem{theorem}{Theorem}[section]
\newtheorem{definition}[theorem]{Definition}

\title[Non-autonomous problem for a $2m$-th order  nonlocal parabolic equation]{Non-autonomous problem for a $2m$-th order semilinear nonlocal parabolic equation}
\author[F. D. M. Bezerra]{Flank D. M. Bezerra}

\address{
Departamento de Matem\'atica\\
Universidade Federal da Para\'iba\\
58051-900 Jo\~ao Pessoa - PB\\
Brazil}
\email{flank@mat.ufpb.br}

\thanks{This work was completed with the support of the grant PID2023-149509N funded by the Ministry of Science, Innovation and Universities of Spain and CNPq/Finance Code \# 303039/2021-3, Brazil funded by the CNPQ of Brazil.
}

\author[S. Sastre-G\'omez]{Silvia Sastre-G\'omez}
\address{Dpto. E.D.A.N.\\
Facultad de Matemáticas\\
Universidad de Sevilla\\
Campus de Reina Mercedes\\
41012 SEVILLA\\
Spain}
\email{ssastre@us.es}

\begin{document}

\maketitle

\begin{abstract}
In this paper we consider a $2m$-th order non autonomous quasilinear parabolic equation. Under suitable conditions of growth and regularity for the nonlinear functions present in the model, we prove a result of existence and characterization of pullback attractors. Moreover, we consider an autonomous version from the  $2m$-th order non autonomous quasilinear parabolic equation in question.


\end{abstract}

\tableofcontents

\section{Introduction}

Let $\Omega\subset\mathbb{R}^N$ be a bounded domain of class $\mathcal{C}^{2m+\mu}$, where $N>2m$, $m\in\mathbb{N}$ and $\mu\in(0,1)$. For scalar function $u=u(x,t)$ defined on $\Omega\times(0,T)$ we denote by $D^\alpha u$ its partial derivative $D_1^{\alpha_1}\ldots D^{\alpha_n}_nu$, where $D_i= \partial_{x_{i}}$, $\alpha=(\alpha_1,\ldots,\alpha_n)$ is multi-indice and $|\alpha|=|\alpha_1|+\ldots+|\alpha_n|$. Inspired by the models with Kirchhoff type diffusion treated in \cite{BCV,CLM-LM,CVV} and \cite{TDT}, we consider the non autonomous quasilinear  evolution 
 equation with $2m$-th order ($m>1$) elliptic operator
\begin{equation}\label{Eq01Local}
\partial_tu(x,t)+a\Big(\int_\Omega ((A+\lambda_0I)u(x,t))u(x,t)dx,t \Big)(A+\lambda_0I)u(x,t)=f(x,t,u),
\end{equation}
with $x\in\Omega$, $t>s$, $s\in\mathbb{R}$, $u=u(x,t)$, where $A$ denotes the differential operator given by
\[
Au=\displaystyle\sum_{|\alpha|,|\beta|\leqslant m} (-1)^{|\beta|}D^\beta(a_{\alpha,\beta}(x)D^\alpha u),
\]
which is a $2m$-th order strongly elliptic differential operator. All the coefficients $a_{\alpha,\beta}$ are of class  $C^{4m+|\beta|+\mu}(\Omega)$ and bounded in $\Omega$ by a constant $c_{up}<0$.

Together with \eqref{Eq01Local}, we consider the initial-boundary conditions
\begin{equation}\label{Condd01}
\begin{cases}
u(x,s)=u_s(x)& x\in\Omega,\ s\in\mathbb{R},\\
B_0u(x,t)=\ldots=B_{m-1}u(x,t)=0& x\in\partial\Omega,\ t>s,
\end{cases}
\end{equation}
where $B_0,\ldots,B_{m-1}$ are linear and time independent boundary operators.

%Standard notation of the functional framework that we use throughout the paper comes from the works \cite{Ch0}, \cite{Ch}, \cite{ChD}, \cite{AF} and \cite{JN}.

The triple $(-A,\{B_j\},\Omega)$ forms a ‘‘regular elliptic boundary value problem'', according to \cite{Ch0,Ch}, \cite[Definition 1.2.1]{ChD}, \cite[Page 76]{AF}, \cite[Page 125]{LM} and \cite{JN}. According to \cite{AF}, it satisfies the ‘‘root condition'', the ‘‘smoothness condition'', certain ‘‘complementary condition'' and the system $\{B_j\}$ is normal. In particular, we need our differential operator $A$ to be uniformly strongly elliptic in $\Omega$; that is, we assume that for some $C>0$
\[
\displaystyle\sum_{|\alpha|,|\beta|=m}a_{\alpha,\beta}(x)\xi^\alpha\xi^\beta\geqslant C|\xi|^{2m},
\]
for any $ x\in\Omega$, $\xi\in\mathbb{R}^N$.

Then  the unbounded linear operator $(A,D(A))$ acting in $L^2(\Omega)$  with domain  
\[
D(A)=\{u\in H^{2m}(\Omega); B_0u=\ldots=B_{m-1}u=0\ \mbox{on}\ \partial\Omega\},
\]
which is the close subspace of the space $H^{2m}(\Omega)$ consisting of functions satisfying, in the sense of traces, boundary conditions $B_j$.  The bilinear form $p_A$ connected to the operator $A$, given by
\[
p_A(u,v)=\displaystyle\sum_{|\alpha|,|\beta|\leqslant m} a_{\alpha,\beta}(x) D^\alpha u D^\beta v,
\]
satisfies the following conditions:

\noindent i) Coerciveness inequality
\begin{equation}\label{coerciveness}
\int_\Omega p_A(w,w)dx+C_2\|w\|_{L^2(\Omega)}^2\geqslant C_1\|w\|^2_{H^{m}(\Omega)}
\end{equation}
for any $w\in H^{m}_0(\Omega)$. 

\noindent ii) Green's Identity
\[
\int_\Omega (Aw)wdx=\int_\Omega p_A(w,w)dx.
\]
for any $w\in H^{m}_0(\Omega)$. 

Therefore, $(A+\omega I,D(A+\omega I))$ with domain $D(A+\omega I)=D(A)$  is a sectorial operator in $L^2(\Omega)$ for each $\omega\in\mathbb{R}$, according to \cite[Example 1.3.8]{ChD}. In particular,  $\lambda_0>0$ may be chosen in \eqref{Eq01Local}, for which  $A+\lambda_0I$ becomes positive in $L^2(\Omega)$, consequently the spectrum of $A+\lambda_0I$ is a point spectrum consisting of eigenvalues $0<\mu_1<\mu_2<\cdots<\mu_{n-1}<\mu_n<\cdots$, with $\lim_{n\to\infty}\mu_n=\infty$, where $\{\mu_n\}_n$ denotes the ordered sequence of eigenvalues of $A+\lambda_0I$,  including their multiplicity.

Let $X^\alpha$ be the fractional power spaces associated to the operator $A$, with $0\leqslant\alpha\leqslant1$. In particular
\[
X^0=L^2(\Omega),
\]
\[
X^{\frac12}=H^{m}_0(\Omega),
\]
and
\[
X^1=D(A)=\{u\in H^{2m}(\Omega); B_0u=\ldots=B_{m-1}u=0\ \mbox{on}\ \partial\Omega\}.
\]

We assume that $a:[0,\infty)\times\mathbb{R}\to[a_0,a_1]$  is  a locally Lipschitz function in the first variable and continuous in the second variable, where $0<a_0<a_1$.

Here, an important fact,  is that the term $a(\int_\Omega (Au)udx,t )$ is non-linear and non-local and its presence in \eqref{Eq01Local} jointly with $f=f(x,t,u)$ turns also 
equation \eqref{Eq01Local} non-autonomous, and all this leading to several interesting and non-trivial mathematical questions about existence of solutions and asymptotic behavior of those solutions. To the best of our knowledge, these models have not yet been addressed in the literature.

On the nonlinearity $f$, let $f:\Omega\times\mathbb{R}\times\mathbb{R}\to \mathbb{R}$ be measurable in the first variable and locally Lipschitz continuous in the second and third variables, uniformly for $x\in\overline{\Omega}$. We assume that there exist $f_0,f_1:\mathbb{R}\to \mathbb{R}$  continuously differentiable such that 
\[
b_0f_0(u)\leqslant f(x,t,u)\leqslant b_1f_1(u),
\]
for any $x\in\Omega$, $t\in\mathbb{R}$, $u\in\mathbb{R}$, for some constants $ b_0<b_1$. Moreover,  there exists a constant $c>0$ such that, for all $u,v\in\mathbb{R}$
\begin{equation}\label{CCondf13333f_i}
|f_i(u)-f_i(v)|\leqslant c|u-v|(|u|^{\rho-1}+|v|^{\rho-1}+1)
\end{equation}
with exponent $\rho$, such that 
\[
1\leqslant\rho\leqslant\dfrac{2N}{N-2m}\quad  \mbox{ with } ~N>2m.
\]
We also assume that
\begin{equation}\label{CCCCCondf2f_i}
uf_i(u)\leqslant -C_0u^2+C_1 u \qquad \mbox{ with } i=0,1,
\end{equation}
 for some $C_0$ such that the first eigenvalue $\lambda_1$ of $A+C_0I$ is positive, and $C_1\geqslant0$.

Moreover, we assume that there exists a constant $c>0$ such that, for all  $x\in\Omega$, $t,s\in\mathbb{R}$ and $u,v\in\mathbb{R}$
\begin{equation}\label{CCondf1}
|f(x,t,u)-f(x,t,v)|\leqslant c|u-v|(|u|^{\rho-1}+|v|^{\rho-1}+1)
\end{equation}
and
\begin{equation}\label{CCondf2}
|f(x,t,u)-f(x,s,u)|\leqslant c|t-s|(|u|^{\rho-1}+1)
\end{equation}
with exponent $\rho$, such that 
\[
1\leqslant\rho\leqslant\dfrac{2N}{N-2m}\quad (N>2m).
\]

Moreover, we also assume that there exist $C_0\in\mathbb{R}$ and $C_1\geqslant0$ such that
\begin{equation}\label{CCCCCondf2}
uf(x,t,u)\leqslant -C_0u^2+C_1 u
\end{equation}
for any $t,u\in\mathbb{R}$ and $x\in\overline{\Omega}$.  More precisely, we assume that \eqref{CCCCCondf2} holds for some $C_0$ such that the first eigenvalue $\lambda_1$ of $A+C_0I$ is positive, as in \eqref{CCCCCondf2f_i}.

The study of the inner structure of  attractors for  semilinear parabolic problems with
local diffusivity has developed considerably and many very interesting results are available
in the literatures, see e.g. \cite{CLR,CLM-LM,SZ} and references therein). The description
of the inner structure for non-local models is much less exploited. Our aim is to unravel the dynamics of the Kirchhoff diffusive type model \eqref{Eq01Local}.

The article is organized as follows. In Section \ref{ComRes} we present comparison results for the problem \eqref{Eq01Local}-\eqref{Condd01} which will allow us to deduce the global solubility of the problem, as well as existence and a result of characterisation of the pullback attractor in $H_0^m(\Omega)$. We also prove, 
using comparison results, the existence of non autonomous equilibrium for \eqref{Eq01Local}-\eqref{Condd01}. Finally, in Section \ref{AutEquat} we prove existence, uniqueness, non triviality and positivity of stationary solution for an autonomous version of \eqref{Eq01Local} in $H_0^m(\Omega)$.

\section{Pullback dynamics}\label{ComRes}
In this section, we simplify the problem by applying a change of variables in time.  Assume that $u$ is a solution of  \eqref{Eq01Local}. We make a change in the variable $t $ as follows 
\[
\phi(t)=\int_0^t a\Big(\int_\Omega  ((A+\lambda_0I)u(x,\sigma))u(x,\sigma) dx,\sigma\Big)d\sigma,
\]
and $w(x,\phi(t))=u(x,t)$, then $w$ satisfies the following equation 
\begin{equation}\label{Eqqqqq}
\partial_tw+(A+\lambda_0I)w=\dfrac{1}{a\Big(\int_\Omega ((A+\lambda_0I)w)wdx,t\Big)}f(x,\phi^{-1}(t),w),
\end{equation}
with $x\in\Omega$, $t>s$, $s\in\mathbb{R}$, subject to the initial-boundary conditions
\begin{equation}\label{Condd02}
\begin{cases}
w(x,\phi(s))=w_s(x)& x\in\Omega,\ s\in\mathbb{R},\\
B_0w(x,\phi(t))=\ldots=B_{m-1}w(x,\phi(t))=0& x\in\partial\Omega,\ t>s,
\end{cases}
\end{equation}
where $w_s(x)=u(x,s)$.
\begin{definition}\label{MildSol01}
We say that $w:[s,s+\tau]\to X^{\frac12}$ is a mild solution of \eqref{Eqqqqq}-\eqref{Condd02} if $w\in C([s,s+\tau],X^{\frac12})$ and $w(t)$ satisfies the variation of constants formula
\[
\begin{split}
&w(t)=e^{-(A+\lambda_0I)(t-s)}w_s\\
&+\int_s^t e^{-(A+\lambda_0I)(t-\xi)} \dfrac{1}{a(\int_\Omega ((A+\lambda_0I)w(\xi))w(\xi) dx,\xi)}f(x,\phi^{-1}(\xi),w(\xi)) d\xi.
\end{split}
\]
\end{definition}

The initial-boundary value problem \eqref{Eqqqqq}-\eqref{Condd02} is locally well-posed in the sense of Definition \ref{MildSol01} and the solutions are jointly continuous with respect
time and initial conditions in $H_0^m(\Omega)$, see \cite{CLR}. By using the comparison result 
we can also prove that the solutions are globally defined in $H_0^m(\Omega)$, see e.g. \cite[Section 6.10]{CLR}. Hence, we can define the solution
operators in the following way: if $w(t,s,w_s)$ is the solution of \eqref{Eqqqqq}-\eqref{Condd02} we write $S(t,s)w_s=w(t,s,w_s)$ and this defines a solution operator family $\{S(t,s);t\geqslant s\in\mathbb{R}\}\subset C(H_0^m(\Omega))$ associated to \eqref{Eqqqqq}-\eqref{Condd02}.

\subsection{Comparison results}

Note that $H_0^m(\Omega)$ is an ordered Banach space with the following partial ordering
structure
\[
u\geqslant v\ \mbox{in}\ H_0^m(\Omega)\ \Longleftrightarrow u(x)\geqslant v(x)\ \mbox{a.e. for}\ x\in\Omega.
\]

Using that order, we can define the positive cone as
\[
\mathcal{C}=\{u\in H_0^m(\Omega);\ u\geqslant 0\}.
\]

Consider the auxiliary  initial-boundary value problem
\begin{equation}\label{AuxProb1}
\begin{cases}
\partial_tz_i(x,t)+(A+\lambda_0I)z_i(x,t)=\dfrac{b_i}{a_j}f_0(z_i(x,t)) &\ x\in\Omega,\ t>0,
%\end{equation}
%subject to the initial-boundary conditions
\\
z_i(x,0)=z^i_0(x) &\ x\in\Omega,
\\
B_1z_i(x,t)=B_2z_i(x,t)=\ldots=B_mz_i(x,t)=0 &\ x\in\partial\Omega,\ t>0,

\end{cases}
\end{equation}
where $i,j\in \{0,1\}$ with $i\ne j$. 
%and the auxiliary  initial-boundary value problem
%\begin{equation}\label{AuxProb2}
%\begin{cases}
%\partial_tv(x,t)+(A+\lambda_0I)v(x,t)=\dfrac{b_1}{a_0}f_1(v(x,t)) &\ x\in\Omega,\ t>0,
%%\end{equation}
%%subject to the initial-boundary conditions
%\\
%v(x,0)=v_0(x) & \ x\in\Omega,
%\\
%B_1v(x,t)=B_2v(x,t)=\ldots=B_mv(x,t)=0 & x\in\partial\Omega,\ t>0.
%\end{cases}
%\end{equation}

\begin{definition}\label{MildSol02}
We say that $z_i:[0,\tau]\to X^{\frac12}$ is a mild solution of \eqref{AuxProb1}  if $z_i\in C([0,\tau],X^{\frac12})$  and $z_i(t)$ satisfies the variation of constants formula
\[
z_i(t)=e^{-(A+\lambda_0I)t}z^i_0+\dfrac{b_0}{a_1}\int_0^t e^{-(A+\lambda_0I)(t-\xi)} f_0(z_i(\xi)) d\xi.
\]
%Analogously, $v:[0,\tau]\to X^{\frac12}$ is a mild solution of \eqref{AuxProb2} if $v\in C([0,\tau],X^{\frac12})$ and $v(t)$ satisfies the variation of constants formula
%\[
%v(t)=e^{-(A+\lambda_0I)t}v_0+\dfrac{b_1}{a_0}\int_0^t e^{-(A+\lambda_0I)(t-\xi)} f_0(v(\xi)) d\xi.
%\]
\end{definition}

It is well know that the semilinear parabolic problem \eqref{AuxProb1} is globally well posed in the sense of Definition \ref{MildSol02}   and if $[0,\infty)\ni t\mapsto z_i(t,z_0)\in H_0^m(\Omega)$ %and  $[0,\infty)\in t\mapsto v(t,v_0)\in H_0^m(\Omega)$   
denotes the solutions of  \eqref{AuxProb1}. %and \eqref{AuxProb2}, respectively, 
We define the semigroups $\{T_0(t);t\geqslant0\}$ and $\{T_1(t);t\geqslant0\}$ by $T_i(t)z^i_0=z_i(t,z^i_0)$ for $i\in\{0,1\}$, see \cite{CLR}.

Back to the problem \eqref{Eqqqqq}, in general we can write the problem \eqref{Eqqqqq} in the abstract form on $L^2(\Omega)$
\begin{equation}\label{Ewqs}
\dfrac{du}{dt}+\varLambda u=g(t,u)
\end{equation}
where 

\noindent [H1] $\varLambda:D(\varLambda)\subset L^2(\Omega)\to L^2(\Omega)$ is the unbounded linear operator defined by 
\[
D(\varLambda)=\{u\in H^{2m}(\Omega); B_0u=\ldots=B_{m-1}u=0\ \mbox{on}\ \partial\Omega\}
\]
 and for each $u\in D(\varLambda)$
 \[
 \varLambda u=(A+\lambda_0I)u. 
 \]
 Let $\rho(\varLambda)$ be the resolvent set of $\varLambda$, it is clear that $(0,\infty)\subset \rho(-\varLambda)$ and that for each element $u_0\in L^2(\Omega)$ with $u_0\geqslant 0$ we have $(\lambda I+\varLambda)^{-1}u_0\geqslant 0$ for any $\lambda>0$. We express this fact by saying that $-\varLambda$ has positive resolvent.
\medskip

\noindent [H2] Consider $g:\mathbb{R}\times H_0^m(\Omega)\to L^2(\Omega)$ a function such that for all $r>0$ we can find $\gamma(r)>0$ such that for all $t\in[t_0,t_1]$, the function $\gamma I_{L^2(\Omega)}+g(t,u)$ is positive for all $u\in\mathcal{C}\cap B^{H_0^m(\Omega)}_r(0)$, where $B^{H_0^m(\Omega)}_r(0)=\{u\in H_0^m(\Omega): \|u\|_{H^{m}_0(\Omega)}<r\}$.

\begin{definition}\label{MildSol03}
We say that $u:[t_0,t_0+\tau]\to X^{\frac12}$ is a mild solution of \eqref{Ewqs}   if $u\in C([t_0,t_0+\tau],X^{\frac12})$ and $u(t)$ satisfies the variation of constants formula
\[
u(t)=e^{-\varLambda (t-t_0)}u_s+\int_{t_0}^t e^{-\varLambda(t-\xi)} g(\xi,u(\xi)) d\xi.
\]
\end{definition}

From now on, we assume that  for $R>0$ we can find $\gamma=\gamma(R)>0$ such that if $|u|\leqslant R$ and $t\in[s,t_0]$ then
\begin{equation}\label{SdfR}
0\leqslant \gamma u+b_0f_0(u)\leqslant \gamma u+g(t,u)\leqslant \gamma u+b_1f_1(u)
\end{equation}
with $\gamma u+b_0f_0(u)$ and $\gamma u+b_1f_1(u)$ increasing. 

Denote by $u_g(t,t_0,u_0)$ the solution of \eqref{Ewqs} in the sense of Definition \ref{MildSol03} for $t\geqslant t_0$  for which the solution is defined. The following theorem provides a comparison result, see \cite[Theorem 6.41]{CLR}.

\begin{theorem}\label{TheoCLR}
If $\varLambda$ is as above and $g,h$, and $\ell$ are functions that satisfies hypothesis [H2]. Then, we have the following.
\begin{enumerate}
\item[(i)]  If for every $r>0$ there is a constant $\gamma=\gamma(r)>0$ such that $\gamma I_{L^2(\Omega)}+g(t,\cdot)$ is increasing in $B^{H_0^m(\Omega)}_r(0)$, for all $t\in[t_0,t_1]$ and $u_0,u_1\in H_0^m(\Omega)$ with $u_0\geqslant u_1$, then 
$$u_g(t,t_0,u_0)\geqslant u_g(t,t_0,u_1)$$ as long as both solutions exist;

\item[(ii)] If $g(t,\cdot)\geqslant \ell(t,\cdot)$ for all $t\in\mathbb{R}$ and $u_0\in H_0^m(\Omega)$, then 
$$u_g(t,t_0,u_0)\geqslant u_\ell(t,t_0,u_0)$$ as long as both solutions exist;

\item[(iii)] If $g,\ell$ are such that for every $r>0$ there is a constant $\gamma=\gamma(r)>0$ and an increasing function $h(t,\cdot)$ such that for every $t\in[t_0,t_1]$ 
\[
g(t,\cdot)+\gamma I_{L^2(\Omega)}\geqslant h(t,\cdot)\geqslant \ell(t,\cdot)+\gamma I_{L^2(\Omega)}
\]
in $B^{H_0^m(\Omega)}_r(0)$ and $u_0,u_1\in H_0^m(\Omega)$ with $u_0\geqslant u_1$, then  
$$u_g(t,t_0,u_0)\geqslant u_g(t,t_0,u_1)$$ as long as both solutions exist.
\end{enumerate}
\end{theorem}

Consequently, we have the following result on the existence of the nonlinear evolution process of solutions associated with the problem  \eqref{Eqqqqq}-\eqref{Condd02}.

\begin{theorem}\label{FirstTheo}
Under the hypotheses in Theorem \ref{TheoCLR}, if $u_0\leqslant u_1\leqslant u_2$ then for each $t\geqslant s\in\mathbb{R}$
\[
T_1(t-s)u_0\leqslant S(t,s)u_1\leqslant T_0(t-s)u_2,
\]
where $T_i(t-s)$ is associated to the problem  \eqref{AuxProb1}, and $S(t,s)$ is associated to the problem \eqref{Eqqqqq}-\eqref{Condd02}
\end{theorem}

\begin{proof} 
Observe that  for $R>0$ we can find $\gamma=\gamma(R)>0$ such that if $|u|\leqslant R$ and $t\in[s,t_0]$ then
\begin{equation}\label{SdfR}
0\leqslant \gamma u+b_0f_0(u)\leqslant \gamma u+g(t,u)\leqslant \gamma u+b_1f_1(u)
\end{equation}
with $\gamma u+b_0f_0(u)$ and $\gamma u+b_1f_1(u)$ increasing. 

We now apply Theorem \ref{TheoCLR} twice to compare solutions of \eqref{Eqqqqq} and \eqref{AuxProb1} with $i=0,1$. To that end, we define
\[
\ell_1(t,u)(x)=\dfrac{1}{a\Big(\int_\Omega ((A+\lambda_0)u)udx,t\Big)}f(\phi^{-1}(t),x,u),
\]
\[
h_1(t,u)(x)=g_1(t,u)(x)=b_1f_1(u),
\]
\[
g_2(t,u)(x)=\dfrac{1}{a\Big(\int_\Omega   ((A+\lambda_0)u)udx,t\Big)}f(\phi^{-1}(t),x,u),
\]
\[
h_2(t,u)(x)=\ell_2(t,u)(x)=b_0f_0(u).
\]

Now, noticing that $H_0^m(\Omega)$ is embedded in $L^\infty(\Omega)$ and using \eqref{SdfR}, we are ready to apply Theorem \ref{TheoCLR}, item $(iii)$ twice to compare solutions of \eqref{Eqqqqq} and \eqref{AuxProb1} with $i=0,1$.  
\end{proof}

\subsection{Pullback attractor}

In this subsection we prove that the evolution process $\{S(t,s);t\geqslant s\in\mathbb{R}\}$ defined by  \eqref{Eqqqqq} admits a pullback attractor in $H_0^m(\Omega)$. 

%{\color{green}
%\begin{theorem}
%Let $S(\cdot,\cdot)$ be an evolution process in a metric space $X$, then the following statements are equivalent:
%
%\noindent $S(\cdot,\cdot)$ has a pullback attractor $\mathcal{A}(\cdot)$;
%
%\noindent There exists a family of compact sets $K(\cdot)$ that pullback attracts bounded subsets of $X$ under $S(\cdot,\cdot)$. In either case 
%\[
%\mathcal{A}(t)=\overline{\bigcup\{\omega(B,t); B\subset X,\ B\ \mbox{is bounded}\}},
%\]
%and $\mathcal{A}(\cdot)$  is minimal in the sense that, if there exists another family of closed bounded sets $\hat{A}(\cdot)$ that pullback attracts bounded subsets of $S$ under $S(\cdot,\cdot)$, then $\mathcal{A}(t)\subset\hat{A}(t)$ for all $t\in\mathbb{R}$.
%\end{theorem}
%
%\begin{theorem}
%Let the introductory estimates \eqref{CCondf13333f_i} and \eqref{CCCCCondf2f_i} be satisfied and the assumptions of Theorem \ref{TheoCLR}. Then the evolution process $\{S(t,s);t\geqslant s\in\mathbb{R}\}$ defined by  \eqref{Eqqqqq} admits a pullback attractor in $H_0^m(\Omega)$.
%\end{theorem}
%
%\begin{proof}
%Thanks to \cite[Theorem 5.3.1]{ChD} it is well know that the problem  \eqref{AuxProb1} has global attractor in $H_0^m(\Omega)$ for any $i=1,2$. ..................
%
%Let $B$ be a bounded subset of  $H_0^m(\Omega)$. Thanks to hypothesis [H1], we have that $\|e^{-(A+\lambda_0 I)(t-s)}\|\le e^{-\delta (t-s)}$. Since $f_i$ for $i\in \{0,1\}$ satisfies hypothesis \eqref{CCCCCondf2f_i}, and thanks to Theorem \ref{FirstTheo}, by standard comparison arguments we have the result. 
%\end{proof}
%}

\medskip

In order to describe the non-autonomous problem pursued in this work we will need to introduce some terminology.

\begin{definition}\label{adsaded3}
Let $Z$ be a Banach space and denote $C(Z)$ the space of continuous functions from $Z$ into $Z$. Let $\{T(t,s);t\geqslant s\in\mathbb{R}\}$ (or $T(\cdot,\cdot)$ ) be a nonlinear evolution process in $C(Z)$.  A positive global solution $\xi$ of $T(\cdot,\cdot)$ is called a {\bf non-autonomous equilibrium} if the zeros of $\xi(t)$ are the same for all $t\in\mathbb{R}$ and $\xi$ is non-degenerate as $t\to\pm\infty$.
\end{definition}

Observe that under the hypotheses mentioned in previous sections is well know that we can find a nontrivial non-negative equilibrium $\phi_{0}^+$ for \eqref{AuxProb1} with $i=0$, and a positive equilibrium $\phi_{1}^+$ for \eqref{AuxProb1} with $i=1$.

Using  Theorem \ref{FirstTheo} and the fact that $T_1(\cdot)$ is a gradient semigroup, we have that
\[
\phi_{1}^+=T_1(t)\phi_{1}^+\leqslant T_0(t)\phi_{1}^+\xrightarrow[t\to+\infty]\ \psi
\]
for some positive equilibrium $\psi$ of \eqref{AuxProb1} with $i=1$.

By the uniqueness of the positive equilibrium for \eqref{AuxProb1} with $i=1$, we conclude that $\psi=\phi_0^+$.

Therefore $\phi_1^+\leqslant \phi_0^+$. Define the set
\[
X_1^+=\{u\in H_0^m(\Omega);\ \phi_1^+(x)\leqslant u(x)\leqslant \phi_0^+ (x)\}.
\]

Now, our idea is to construct a positive non-autonomous equilibrium for \eqref{Eq01Local}, see Definition \ref{adsaded3}. For that, we will prove that $X_1^+$ is positively invariant, that is, $S(t,s)X_1^+\subset X_1^+$, for all $t\geqslant s\in\mathbb{R}$.
%
%
%Now, our idea is to construct a positive non-autonomous equilibrium for \eqref{Eq01Local}, see Definition \ref{adsaded3}.
%
For $x\in\Omega$ and $u_0\in X_1^+$ since $T_i(t-s)\phi_i^+=\phi_i^+$, $i=0,1$, we have
\[
\phi_1^+(x)\leqslant T_0(t-s)u_0\leqslant S(t,s)u_0\leqslant T_1(t-s)u_0\leqslant \phi_0^+(x)
\]

If $u(t,s,u_0)(x)=S(t,s)u_0(x)$, since $u_0$ is constant  in $\Omega$, then  we conclude that $u(t,s,u_0)(x)$ is constant in $\Omega$.

\begin{theorem} 
The evolution process $S(\cdot,\cdot)$ restricted to $X_1^+$ admits a pullback attractor. In particular, there exists a non-autonomous equilibrium in $\mathcal{C}$.
\end{theorem}

\begin{proof}
The invariance follows from the reasoning that preceded the theorem. The fact that $S(\cdot,\cdot)$ has a pullback attractor in $H_0^m(\Omega)$ ensures that it also has a pullback attractor when restricted to $X_1^+$. Now, any global solution in the pullback attractor of $S(\cdot,\cdot)$ restricted to $X_1^+$ is a non-autonomous equilibria.
\end{proof}

\section{The autonomous equation}\label{AutEquat}

In this section  we consider \eqref{Eq01Local} with $a$ and $f$ independent of time. Namely, we consider the autonomous problem
\begin{equation}\label{Eq01LocalSemt}
\partial_tu+a\Big(\int_\Omega ((A+\lambda_0I)u)udx \Big)(A+\lambda_0I)u=f(x,u),
\end{equation}
for any $x\in\Omega$, $t>0$, subject to initial-boundary conditions given in \eqref{Condd01}.

An equilibrium solution of the problem \eqref{Eq01LocalSemt} is an solution independent of time; namely, is a solution of the 
elliptical problem
\begin{equation}\label{Eq01Lmt}
a\Big(\int_\Omega ((A+\lambda_0I)u)udx \Big)(A+\lambda_0I)u=f(x,u).
\end{equation}

From now on we necessity clarify that:
\begin{definition}
We say that $u\in H_0^m(\Omega)$ is a weak solution of the problem \eqref{Eq01Lmt} if 
\begin{equation}\label{EqquatNonlocal}
a\Big(\int_\Omega ((A+\lambda_0I)u)(x)u(x)dx \Big)\int_\Omega (A+\lambda_0I) u(x)v(x)dx=\int_\Omega f(x,u(x))v(x)dx,
\end{equation}
for all $v\in H_0^m(\Omega)$.
\end{definition}

From now on  we study the existence of a positive, nontrivial and  equilibria to the autonomous problem \eqref{Eq01LocalSemt} (namely, a weak solution of \eqref{Eq01Lmt} in the sense of the above definition) by minimizing the energy if we consider $a:[0,+\infty)\to\mathbb{R}$ is a continuous function such that
\begin{equation}\label{Eaaaat}
0<\sigma\leqslant a(t)
\end{equation}
for any $t\in\mathbb{R}$.

Let 
\[
F(x,u)=\int_0^uf(x,\theta)d\theta
\]
be the primitive of $f$. We assume that for some $R>0$
\begin{equation}\label{EDswa}
0<\mu F(x,u)\leqslant f(x,u)u,
\end{equation}
 for any $|u|\geqslant R$.

Since $a$ is continuous and $f$ has subcritical growth, the nonlinear functional $E:H_0^m(\Omega)\to\mathbb{R}$ defined by
\begin{equation}\label{Eq01LocalEEESemt}
E(u)=\frac{1}{2}\int_0^{\int_\Omega ((A+\lambda_0I)u)udx} a(s)ds-\int_\Omega F(x,u(x))dx
\end{equation}
is of class $C^1$ in $H_0^m(\Omega)$; namely, $E$ is a Lyapunov function for \eqref{Eq01LocalSemt}, and as a matter of fact, by combining the growth of $a$ and $f$, we can easily obtain existence results of nontrivial weak solution of the problem \eqref{Eq01Lmt} by minimization arguments, see e.g. \cite{ACMa}. Hence, the nonlinear semigroup associated with \eqref{Eq01LocalSemt}-\eqref{Condd01} on $H_0^m(\Omega)$  is gradient.

Note that a solution of \eqref{EqquatNonlocal} can be founded as critical point of the energy functional $E: H_0^m(\Omega)\to\mathbb{R}$ defined by \eqref{Eq01LocalEEESemt}.

Note that 
\[
\begin{array}{ll}
\langle E'(u),v \rangle&\displaystyle = a\Big(\int_\Omega ((A+\lambda_0I)u)(x)u(x)dx \Big)\int_\Omega((A+\lambda_0I)u)(x)v(x)dx
\smallskip\\
&\displaystyle -\int_\Omega f(x,u(x))v(x)dx,
\end{array}
\] 
for all $u,v\in H_0^m(\Omega)$.

First, we can prove that $E$ is a coercive functional on $H_0^m(\Omega)$ as a consequence of \eqref{CCCCCondf2}, \eqref{EDswa} and
\[
\begin{split}
E(u)&\geqslant \dfrac{\sigma}{2}\int_\Omega ((A+\lambda_0I)u)(x)u(x)dx+\dfrac{1}{\mu}\int_\Omega ( C_0|u|^2-C_1 )dx\\
&\geqslant \dfrac{\sigma \,C_1}{2}\|u\|^2_{H^{m}(\Omega)}+\Big(\dfrac{C_0}{\mu}-\dfrac{\sigma \,C_2}{2}\Big)\|u\|_{L^2(\Omega)}^2-\dfrac{C_1}{\mu}|\Omega|\\
&\geqslant C_\mu\|u\|^2_{H^{m}(\Omega)}-\dfrac{C_1}{\mu}|\Omega|,
\end{split}
\]
for some $C_\mu>0$, where $\mu>0$ sufficiently large. Also, $E$ is weakly lower semicontinuous on $H_0^m(\Omega)$ thanks to the Lebesgue’s Theorem on dominated convergence. Thus we have the following result.

From now on, we assume that 
\[
a(0)\mu_1<c_0,
\]
where $c_0>0$ is such that 
\[
F(x,u)\geqslant -c_0|u|^2 + \zeta(u).
\]

\begin{theorem}
Suppose that the application $t \mapsto a(t)$ is only continuous and non-decreasing and \eqref{Eaaaat} holds. Moreover, assume that  $\overline{u}\in H^m_0(\Omega)$ is a supersolution to problem \eqref{Eq01Lmt}; that is, $\overline{u}\geqslant0$ on $\partial\Omega$ and for any $\varphi\in C_0^\infty(\Omega)$ with  $\varphi\geqslant0$ we have
\[
\begin{array}{ll}
a\Big(&\displaystyle\int_\Omega ((A+\lambda_0I)\overline{u})(x)\overline{u}(x)dx \Big)\int_\Omega((A+\lambda_0I)\overline{u})(x)\varphi(x)dx
\\
&\displaystyle-\int_\Omega f(x,\overline{u}(x))\varphi(x)dx\geqslant0.
\end{array}
\]
 and assume that with constant $\overline{c}\in\mathbb{R}$ there holds $0\leqslant\overline{u}\leqslant \overline{c}<\infty$, almost everywhere in $\Omega$. Then there exists a weak solution $u\in H^m_0(\Omega)$ of \eqref{Eq01Lmt}, satisfying the condition $0\leqslant u\leqslant \overline{u}$ almost everywhere in $\Omega$. In other words,  there exists a  positive and non-trivial weak solution of the problem \eqref{Eq01Lmt}.
\end{theorem}

\begin{proof}
Since we are interested in positive and  weak solutions for the problem  \eqref{Eq01Lmt}, we restrict the domain of $E$ to 
\[
\mathcal{M}=\{u\in H_0^m(\Omega); 0\leqslant u\leqslant \overline{u}\ \mbox{almost everywhere in}\ \Omega\}.
\]
Here, $u_C\equiv C>0$ is a (weak) supersolution of  \eqref{Eq01Lmt}; namely, 

Clearly $\mathcal{M}$ is weakly closed (closed and convex), that is $\mathcal{M}$ and $H_0^m(\Omega)$ satisfies all the conditions of the \cite[Theorem 1.2]{MS}. Moreover, $F(x,u(x))$ is bounded for all $u\in\mathcal{M}$ and almost every $x\in\Omega$. Then we infer the existence of the relative minimizers $u\in\mathcal{M}$. To show that $u$ is a weak solution of    \eqref{Eq01Lmt}, we consider $\varphi\in C_0^\infty(\Omega)$ and $\epsilon>0$. Let 
\[
v_\epsilon=\min\{\overline{u},\max\{0,u+\epsilon\varphi\}\}=u+\epsilon\varphi-\varphi^\epsilon+\varphi_\epsilon\in \mathcal{M},
\]
where
\[
\varphi^\epsilon=\max\{0,u+\epsilon\varphi-\overline{u}\}\geqslant0,
\]
and
\[
\varphi_\epsilon=\max\{0,-(u+\epsilon\varphi)\}\geqslant0,
\]
we note that $\varphi^\epsilon,\varphi_\epsilon\in H_0^m(\Omega)\cap L^\infty(\Omega)$.

The functional $E$ is differentiable in direction $v_\epsilon-u$.  Since $u$ minimizes $E$ in $\mathcal{M}$ we have
\[
0\leqslant \langle E'(u), v_\epsilon-u\rangle=\epsilon \langle E'(u), \varphi\rangle - \langle E'(u), \varphi^\epsilon\rangle +\langle E'(u), \varphi_\epsilon\rangle, 
\]
so that 
\[
 \langle E'(u), \varphi\rangle\geqslant \dfrac{1}{\epsilon} [ \langle E'(u), \varphi^\epsilon\rangle -\langle E'(u), \varphi_\epsilon\rangle] 
\]

Now, since $\overline{u}$ is a supersolution to \eqref{Eq01Lmt}, we have
\[
\begin{split}
\langle E'(u), \varphi^\epsilon\rangle&=\langle E'(\overline{u}), \varphi^\epsilon\rangle+\langle E'(u)-E'(\overline{u}), \varphi^\epsilon\rangle\\
&\geqslant \langle E'(u)-E'(\overline{u}), \varphi^\epsilon\rangle\\
&\geqslant \langle E'(u-\overline{u}), \varphi^\epsilon\rangle,
\end{split}
\]
and so
\[
\begin{array}{l}
\langle E'(u), \varphi^\epsilon\rangle
\smallskip\\
\displaystyle =a\Big(\int_\Omega ((A+\lambda_0I)(u-\overline{u}))(u-\overline{u})dx\Big)\int_{\Omega_\epsilon} ((A+\lambda_0I)(u-\overline{u})(u+\epsilon\varphi-\overline{u}) dx
\smallskip\\
\displaystyle-\int_{\Omega_\epsilon}  \left(f(x,u)-f(x,\overline{u})\right)\varphi^\epsilon dx
\smallskip\\
\displaystyle=a\Big(\int_\Omega ((A+\lambda_0I)(u-\overline{u}))(u-\overline{u})dx\Big)\int_{\Omega_\epsilon} ((A+\lambda_0I)(u-\overline{u})(u+\epsilon\varphi-\overline{u}) dx
\smallskip\\
\displaystyle-\int_{\Omega_\epsilon}  \left(f(x,u)-f(x,\overline{u})\right)(u+\epsilon\varphi-\overline{u}) dx
\end{array}
\]
that is,
\[
\begin{array}{l}
\langle E'(u), \varphi^\epsilon\rangle
\smallskip\\
\displaystyle\geqslant a\Big(\int_\Omega ((A+\lambda_0I)(u-\overline{u}))(u-\overline{u})dx\Big)\int_{\Omega_\epsilon} ((A+\lambda_0I)(u-\overline{u})(u+\epsilon\varphi-\overline{u}) dx
\smallskip\\
\displaystyle-\epsilon\int_{\Omega_\epsilon}  |f(x,u)-f(x,\overline{u})||\varphi| dx
\end{array}
\]

Consequently
\[
\begin{array}{l}
\langle E'(u), \varphi^\epsilon\rangle
\smallskip\\
\displaystyle\geqslant \epsilon a\Big(\int_\Omega ((A+\lambda_0I)(u-\overline{u}))(u-\overline{u})dx\Big) \int_{\Omega_\epsilon} ((A+\lambda_0I)(u-\overline{u})(u+\epsilon\varphi-\overline{u}) dx
\smallskip\\
\displaystyle-\epsilon\int_{\Omega_\epsilon}  |f(x,u)-f(x,\overline{u})||\varphi| dx
\smallskip\\
\displaystyle\geqslant \epsilon a_0\int_{\Omega_\epsilon} ((A+\lambda_0I)(u-\overline{u})(u+\epsilon\varphi-\overline{u}) dx-\epsilon\int_{\Omega_\epsilon}  |f(,u)-f(,\overline{u})||\varphi| dx
\smallskip\\
\displaystyle\geqslant \epsilon a_0\int_{\Omega_\epsilon} ((A+\lambda_0I)(u-\overline{u})(u+\epsilon\varphi-\overline{u}) dx
\smallskip\\
\displaystyle
-\epsilon\|\varphi\|_{L^\infty(\Omega)}\int_{\Omega_\epsilon}  |f(x,u)-f(x,\overline{u})| dx
\end{array}
\]
where $\Omega^\epsilon=\{x\in\Omega; u(x)+\epsilon\varphi(x)\geqslant \overline{u}(x)>u(x)\}$, that satisfies $|\Omega^\epsilon|\to0$ as $\epsilon\to0$. Here, $|\Omega^\epsilon|$ denotes the Lebesgue’s measure of $\Omega^\epsilon$.  Hence by absolute continuity of the Lebesgue integral we obtain that
\[
\langle E'(u), \varphi^\epsilon\rangle\geqslant o(\epsilon),
\]
where $o(\epsilon)/\epsilon\to 0$ as $\epsilon\to0$.  Similarly
\[
\langle E'(u), \varphi_\epsilon\rangle\leqslant o(\epsilon),
\]
whence
\[
\langle E'(u), \varphi\rangle\geqslant 0
\]
for all $\varphi\in C_0^\infty(\Omega)$. Reversing the sign of $\varphi$ and since $ C_0^\infty(\Omega)$ is dense in $H^m_0(\Omega)$ we finally see that $E'(u)=0$, as claimed.

Now we will to show that $u$ is a non-trivial. For that, we will show that the minimum of energy is negative what it guarantees that $u$ could not be zero.  In fact, let $\phi$ be the eigenfunction associated to the first eigenvalue $\mu_1$ of the operator $A$ on $H_0^m(\Omega)$; that is, $\phi$ is is solution to the following eigenvalue problem
\[
(A+\lambda_0I)\phi=\mu_1\phi\ \ \mbox{in}\ \ \Omega.
\]

Since that $a(0)\mu_1<c_0$ from continuity of the function $a$, we have $a(t)\mu_1< c_0$ for each $t\in\Big[0,\,\mu_1\delta^2\int_\Omega|\phi|^2dx\Big]$ for some $\delta>0$ small enough. Note that $\delta\phi\in\mathcal{M}$, thus
\[
\begin{split}
E(\delta\phi)&=\dfrac{1}{2}\int_0^{\int_\Omega ((A+\lambda_0I)\delta\phi)\delta\phi dx} a(s)ds-\int_\Omega F(x,\delta\phi(x))dx\\
&=\dfrac{1}{2}a(c_\delta)\mu_1\delta^2 \int_\Omega|\phi|^2dx   -\int_\Omega F(x,\delta\phi(x))dx\\
\end{split}
\]
for some $c_\delta\in\Big[0,\,\mu_1\delta^2\int_\Omega|\phi|^2dx\Big]$, and consequently
\[
\begin{split}
E(\delta\phi)&\leqslant \dfrac{\delta^2}{2}\Big[[a(c_\delta)\mu_1-c_0]  \int_\Omega|\phi|^2dx +c_1\zeta(\phi)\Big]\\
&<0.
\end{split}
\]

Thus
\[
E(u_0)=E_\infty=\inf_{u\in H_0^m(\Omega)}E(u)\leqslant E(\delta\phi)<0
\]
for $\delta$.
\end{proof}

%Finally, assume, by contradiction, that $u$ and $v$ are two distinct nontrivial non-negative solutions of \eqref{Eq01Lmt}. Let $\{x\in\Omega_1; u(x)>v(x)\}$. Assume that $\Omega_1$ is not empty. It is clear that $u=v$ on $\partial\Omega_1$. By regularity we have $u,v\in C^m(\Omega)$ and by the maximum principle $u,v>0$ in $\Omega$, see \cite{MHP-HFW}, so that 
%\[
%a\Big(\int_\Omega (Au)(x)u(x)dx \Big)\int_\Omega(Au)(x)(u-v)(x)dx=\int_\Omega f(x,u(x))(u-v)(x)dx,
%\]
%and
%\[
%a\Big(\int_\Omega (Av)(x)v(x)dx \Big)\int_\Omega(Av)(x)(u-v)(x)dx=\int_\Omega f(x,v(x))(u-v)(x)dx,
%\]
%
%\[
%\begin{split}
%&a\Big(\int_\Omega (Au)udx \Big)\int_\Omega(Au)(u-v)dx\pm a\Big(\int_\Omega (Au)udx \Big)\int_\Omega(Av)(u-v)dx\\
%&=\int_\Omega [f(\cdot,u)-f(\cdot,v)](u-v)dx,
%\end{split}
%\]
%and consequently
%\[
%\begin{split}
%&a\Big(\int_\Omega (Au)udx \Big)\int_\Omega p_A(u-v,u-v)dx+\Big[ a\Big(\int_\Omega (Au)udx \Big)- a\Big(\int_\Omega (Av)vdx \Big)\Big]\int_\Omega(Av)(u-v)dx\\
%&=\int_\Omega [f(\cdot,u)-f(\cdot,v)](u-v)dx,
%\end{split}
%\]
%
%\[
%a_0C_1\|u-v\|^2_{H^{m}(\Omega)}-a_0C_2\|u-v\|_{L^2(\Omega)}^2\leqslant 
%\]
%\end{proof}

\end{document}